\documentclass[12pt]{amsart}
\usepackage{amssymb,latexsym}
\usepackage{enumerate}

\makeatletter
\@namedef{subjclassname@2010}{%
  \textup{2010} Mathematics Subject Classification}
\makeatother

\newtheorem{thm}{Theorem}[section]
\newtheorem{cor}[thm]{Corollary}

\theoremstyle{definition}
\newtheorem{defn}[thm]{Definition}
\newtheorem{remark}[thm]{Remark}
\newtheorem{ex}[thm]{Example}

\numberwithin{equation}{section}

\def\ds{\displaystyle}

\def\Z {{\mathbb Z}}

\def\Q {{\mathbb Q}}

\def\C {{\mathcal C}}

\def\red#1 {\textcolor{red}{#1 }}
\def\blue#1 {\textcolor{blue}{#1 }}

\begin{document}

\title[Using Lucas Sequences to Generalize a Theorem of Sierpi\'{n}ski]{Using Lucas Sequences to Generalize a Theorem of Sierpi\'{n}ski}
\author[Lenny Jones]{Lenny Jones}
\address{Department of Mathematics, Shippensburg University, Pennsylvania, USA}
\email{lkjone@ship.edu}

\date{\today}

\begin{abstract}
In 1960, Sierpi\'{n}ski proved that there exist infinitely many odd positive integers $k$ such that $k\cdot 2^n+1$ is composite for all positive integers $n$. In this paper, we prove some generalizations of Sierpi\'{n}ski's theorem with $2^n$ replaced by expressions involving certain Lucas sequences $U_n(\alpha,\beta)$. In particular, we show the existence of infinitely many Lucas pairs $(\alpha,\beta)$, for which there exist infinitely many positive integers $k$, such that $k \left(U_n(\alpha,\beta)+(\alpha-\beta)^2\right)+1$ is composite for all integers $n\ge 1$. Sierpi\'{n}ski's theorem is the special case of $\alpha=2$ and $\beta=1$. Finally, we establish a nonlinear version of this result by showing that there exist infinitely many rational integers $\alpha>1$, for which there exist infinitely many positive integers $k$, such that $k^2 \left(U_n(\alpha,1)+(\alpha-1)^2\right)+1$ is composite for all integers $n\ge 1$.
\end{abstract}


\maketitle

\section{Introduction}\label{Intro}



The following concept, originally due to Erd\"os \cite{Ecov}, is crucial to all results in this article.
\begin{defn}\label{defncov}
  A {\it covering} of the integers is a system of congruences $x\equiv r_i \pmod{m_i}$ such that every integer satisfies at least one of the congruences. A covering is said to be a {\it finite covering} if the covering contains only finitely many congruences.
\end{defn}
\begin{remark}
Since all coverings in this paper are finite coverings, we omit the word ``finite".
\end{remark}
In 1960, using a particular covering, Sierpi\'{n}ski \cite{Sier} published a proof of the fact that there exist infinitely many odd positive integers $k$ such that
$k\cdot 2^n+1$ is composite for all natural numbers $n$. Any such value of $k$ is called a Sierpi\'{n}ski number. Since then, several authors \cite{Chen1,Chen3,Chen2,Chen4,Chen5,FFK,FM04,F1,F2,Izotov,Lenny1,Lenny2} have investigated generalizations and variations of this result. We should also mention a paper of Riesel \cite{Riesel}, which actually predates the paper of Sierpi\'{n}ski, in which Riesel proves a similar result for the sequence of integers $k\cdot 2^n-1$.
We give a proof of Sierpi\'{n}ski's original theorem since it provides an easy introduction to the techniques used in this paper.
\begin{thm}[Sierpi\'{n}ski \cite{Sier}]\label{Thm:Sierpinski}
There exist infinitely many odd positive integers $k$ such that
$k\cdot 2^n+1$ is composite for all integers $n\ge 1$.
\end{thm}
\begin{proof}
Consider the following covering $n\equiv r_i \pmod{m_i}$:
\begin{center}
$\begin{array}{ccccccccc}
  i & & 1&2&3&4&5&6&7\\ \hline
  r_i & & 1 & 2 & 4 & 8 & 16 & 32 & 0\\
 m_i & & 2 & 4 & 8 & 16 & 32 & 64 & 64.
  \end{array}$
 \end{center}
For each $i$, when $n\equiv r_i \pmod{m_i}$ and $k\equiv b_i
\pmod{p_i}$ (from below),
\begin{center}$\begin{array}{ccccccccc}
i & & 1&2&3&4&5&6&7\\ \hline
  b_i & & 1 & 1 & 1 & 1 & 1 & 1 & -1\\
  p_i & & 3 & 5 & 17 & 257 & 65537 & 641 & 6700417,
  \end{array}$
  \end{center}
  it is easy to check that $k\cdot 2^n+1$ is divisible by $p_i$. Now, apply the Chinese Remainder Theorem
  to the system $k\equiv b_i \pmod{p_i}$.
   Then, for any integer $n\ge 0$, and any such solution $k$, we have that $k\cdot 2^n+1$ is divisible by at least one prime from the set $\left\{ 3,5,17,257,641,65537,6700417 \right\}$.
\end{proof}
This paper is concerned with generalizations of Theorem \ref{Thm:Sierpinski} which involve Lucas sequences. 
A pair $\left(\alpha,\beta\right)$ of algebraic integers, where $\alpha +\beta$ and $\alpha \beta$ are nonzero relatively prime rational integers, and $\alpha/ \beta$ is not a root of unity, is called a {\it Lucas pair}. For $n\ge 0$, we can then define a sequence of rational integers
\[U_n(\alpha,\beta):=\frac{\alpha^n-\beta^n}{\alpha- \beta}.\]
 When the values of $\alpha$ and $\beta$ are general, or they are clear from the context of the discussion, we simply write $U_n$. Such a sequence is known as a {\it Lucas sequence of the first kind}. Unless otherwise stated, we assume throughout this paper, without loss of generality, that $\alpha>\beta$. The observation that
\[2^n=U_n\left(2,1\right)+(2-1)^2\]
provides the motivation for the results in this paper. We replace $2^n$ with $U_n\left(\alpha,\beta\right)+(\alpha-\beta)^2$, and investigate when there exist infinitely many values of $k$ such that the sequence
\[k\left(U_n\left(\alpha,\beta\right)+(\alpha-\beta)^2\right)+1\] is composite for all integers $n\ge 1$.

 This paper is organized as follows. In Section \ref{Section:I}, our focus is on the Lucas pairs $\left(\alpha, \beta\right)$, where $\alpha$ is a rational integer and $\beta=1$.
In Section \ref{Section:II}, we consider Lucas pairs $\left(\alpha, \beta\right)$, where $\alpha$ and $\beta$ are not necessarily rational. The rational Lucas pairs $\left(\alpha, \beta\right)$ covered in Section \ref{Section:II} have the property that $\alpha-\beta=1$. Although the Lucas pair $(2,1)$ from Theorem \ref{Thm:Sierpinski} falls into this category, the actual method used in the proof of Theorem \ref{Thm:general} in Section \ref{Section:II} does not ``capture" this particular Lucas pair. However, the technique can be modified to achieve this goal. In Section \ref{Section:III}, we develop a more general approach. Theoretically, the techniques there can be used for any Lucas pair. However, it is difficult to categorize the Lucas pairs for which the methods will actually be successful.


A key idea in the proof of Theorem \ref{Thm:Sierpinski} is the availability of enough useful primes: a unique prime $p$ corresponding to each congruence in the covering that allows us to reduce the power $2^n$ modulo $p$ to an unambiguous value. In all main theorems in this article, we need such a set of primes. However, the way these primes are ``manufactured" is quite different in each section. Our approach in Section \ref{Section:I} is more typical of theorems of this nature. We use the concept of a {\it primitive divisor}; see Section \ref{Section:I} for a full explanation. But in Sections \ref{Section:II} and \ref{Section:III}, the methods used to produce the desired primes appear to be new. In fact, it seems unlikely that traditional applications of primitive divisors could be used to prove the results in Sections \ref{Section:II} and \ref{Section:III}.

\section{Generalization I}\label{Section:I}
In this section, we present a generalization of Theorem \ref{Thm:Sierpinski}, whose proof utilizes the concept of a primitive divisor.
As previously mentioned, primitive divisors are often useful in proving theorems similar to Theorem \ref{Thm:Sierpinski}, and various related applications \cite{Chen2,Chen4,Chen5,Chen2009a,Lenny1,Lenny2,Sun2009a,Sun1998a}.

\begin{defn}\label{Def:primdiv}
 For any Lucas pair $\left( \alpha, \beta \right)$, we define a {\it primitive (prime) divisor} of $U_n$ to be a prime $p$ such that both of the following conditions hold:
\begin{itemize}
  \item $U_n \equiv 0 \pmod{p}$,
  \item $\left(\alpha-\beta\right)^2U_1U_2\cdots U_{n-1} \not \equiv 0 \pmod{p}$.
\end{itemize}
We say that the Lucas pair $\left(\alpha, \beta \right)$ is {\it $n$-defective} if $U_n$ has no primitive divisor.
\end{defn}

 The following result, originally due to Zsigmondy \cite{Z}, provides us with conditions in certain situations under which these primitive divisors exist.
\begin{thm}\label{Thm:Zsig}
Let $\alpha$ and $\beta$ be coprime positive rational integers, and let $n$ be a positive integer. Then there exists at least one prime $p$ such that $\alpha^n-\beta^n\equiv 0 \pmod{p}$ and $\alpha^m-\beta^m\not \equiv 0 \pmod{p}$ for all positive integers $m<n$, with the following exceptions:
\begin{itemize}
\item $\alpha=2$, $\beta=1$ and $n=6$
\item $\alpha+\beta$ is a power of $2$ and $n=2$.
\end{itemize}
\end{thm}
Theorem \ref{Thm:Zsig} is a generalization of the case when $\beta =1$, which is due to Bang \cite{Bang}. Birkhoff and Vandiver \cite{BV} proved Theorem \ref{Thm:Zsig} independently of Zsigmondy.

%

From Definition \ref{Def:primdiv}, the following corollary of Theorem \ref{Thm:Zsig} is immediate.
\begin{cor}\label{Cor:Zsig}
Let $(\alpha,\beta)$ be a rational Lucas pair, and let $n\ge 2$ be an integer. Then $U_n$ has a primitive divisor, with the only exceptions being the exceptions noted in Theorem \ref{Thm:Zsig}.
\end{cor}


The situation when $\alpha$ and $\beta$ are not rational integers is much more difficult. Early work was done by Carmichael \cite{Car}, Ward \cite{Ward} and Voutier \cite{VoutierI}. More recently, using deep ideas from transcendence theory, Bilu, Hanrot and Voutier \cite{BHV} have shown, for any Lucas pair $(\alpha,\beta)$, that $U_n$ has a primitive divisor for all $n\ge 30$, and they have determined all $n$-defective Lucas pairs.

In this section, our focus is on Lucas pairs $\left(\alpha, 1\right)$, where $\alpha$ is a rational integer. The approach is somewhat conventional, in that we use a covering that is constructed by means of primitive divisors. However, complications arise in the proof, and we are forced to show the existence of a second primitive divisor in certain situations. In general, it is still a mystery as to exactly when $U_n$ possesses a second primitive divisor. The best known results in this direction, when $\alpha$ and $\beta$ are rational, are due to Schinzel \cite{Schinzel1962prim}, but unfortunately, they are not applicable in all of our situations.
We state below, without proof, some well-known results that relate these particular Lucas sequences to values of certain cyclotomic polynomials. These facts are useful here to help establish the existence of a second primitive divisor. We let $\Phi_n(x)$ denote the $n$-th cyclotomic polynomial, and $U_n$ denote $U_n(\alpha,1)$, where $\alpha\ge 2$ is an integer.

The following theorem is due to Legendre \cite{RBook}.
\begin{thm} \label{Thm:Legendre}
  Let $q$ be a prime divisor of $\Phi_n(\alpha)$, and let $ord_q(\alpha)$ denote the order of $\alpha$ modulo $q$. If $ord_q(\alpha)<n$, then $q$ divides $n$. 
\end{thm}

Since \[(\alpha-1)U_n=\prod_{d|n}\Phi_d(\alpha),\]

the following corollary is immediate from
Theorem \ref{Thm:Legendre}.
\begin{cor}\label{Cor:primdiv}\text{}
\begin{enumerate}
\item \label{Cor:2} A prime divisor $q$ of $U_n$ is a primitive divisor of $U_n$ if and only if $\Phi_m(\alpha)\not \equiv 0 \pmod{q}$ for all proper divisors $m$ of $n$.
\item \label{Cor:3} If $q$ is a primitive divisor of $U_n$, then $\Phi_n(\alpha)\equiv 0 \pmod{q}$.
\item \label{Cor:4} If $\Phi_n(\alpha)\equiv 0 \pmod{q}$ and $n\not \equiv 0 \pmod{q}$, then $q$ is a primitive divisor of $U_n$.
\end{enumerate}
\end{cor}

 The main result of this section is the following:
 \begin{thm} \label{Thm:main}
  Let $\alpha \ge 2$ be a rational integer. Then
 there exist infinitely many positive integers $k$ such that
    \[k \left(U_n(\alpha,1)+(\alpha-1)^2\right)+1\] is 
     composite for all integers $n\ge 1$.

\end{thm}
\begin{proof}
Since $\alpha=2$ corresponds to Sierpi\'{n}ski's original theorem, we assume that $\alpha\ge 3$. Note that, by Corollary \ref{Cor:Zsig}, the only $n$-defective pairs that are of concern to us here are $(\alpha,\beta)=(2^c-1,1)$, which are $2$-defective. So, the proof is broken into two main cases: $\alpha\ne 2^c-1$ and $\alpha=2^c-1$. A different covering  $\left\{n\equiv r_i \pmod{m_i}\right\}$ is used in each of these cases. We use a covering with minimum modulus 3 in the case when $\alpha=2^c-1$, to circumvent the fact that these sequences are 2-defective. In both cases, we let $p_i$ denote a primitive divisor of $\ds U_{m_i}$. Then, when $n\equiv r_i \pmod{m_i}$, we have
\begin{align*}
U_{n}+(\alpha-1)^2&\equiv U_{r_i}+(\alpha-1)^2\pmod{p_i}
\end{align*}
For brevity of notation, we define $A_i:=U_{r_i}+(\alpha-1)^2$. It is crucial for our arguments that $A_i$ be invertible modulo $p_i$. In other words, we need $A_i\not \equiv 0\pmod{p_i}$ for all $i$.

Assume first that $\alpha \ne 2^c-1$, 
and use the covering:
\begin{center}
$\begin{array}{ccccccc}
  i & & 1&2&3&4&5\\ \hline
  r_i & & 0 & 0 & 1 & 1 & 11\\
 m_i & & 2 & 3 & 4 & 6 & 12.
  \end{array}$
 \end{center}

Next, we verify 
 for each $i\ne 3$, that $A_i\not \equiv 0 \pmod{p_i}$.
This is clear when $i=1,2$, since
$A_i\equiv  (\alpha-1)^2 \not \equiv 0 \pmod{p_i}$.

For $i=4$, we have $A_4=\alpha^2-2\alpha+2$.
Since $p_4$ is a primitive divisor of $U_6$, part (\ref{Cor:3}) of Corollary \ref{Cor:primdiv} tells us that $\alpha^2-\alpha+1\equiv 0 \pmod{p_4}$. If $A_4\equiv 0 \pmod{p_4}$, then \[0\equiv \alpha^2-\alpha+1-(\alpha^2-2\alpha+2)\equiv \alpha-1 \pmod{p_4},\]
which contradicts the fact that $p_4$ is primitive.

Now consider $i=5$. Since $p_5$ is a primitive divisor of $U_{12}$, we have, by part (\ref{Cor:3}) of Corollary \ref{Cor:primdiv}, that
$\alpha^6\equiv -1 \pmod{p_5}$. Using this fact, it is easy to show that
\[0\equiv (-2\alpha^3-\alpha^2+2\alpha)A_5\equiv 5(\alpha-1) \pmod{p_5}.\]
Hence, $p_5=5$, since $\alpha-1\not \equiv 0 \pmod{p_5}$.
Thus, \[\Phi_{12}(\alpha)=\alpha^4-\alpha^2+1\equiv 0 \pmod{5},\] from part (\ref{Cor:3}) of Corollary \ref{Cor:primdiv}. But then, since $\alpha \not \equiv 0 \pmod{p_5}$, we arrive at the contradiction that 2 is a square modulo 5.

Finally, to finish the proof when $\alpha\ne 2^c-1$, we examine the case of $i=3$. Suppose that $A_3\equiv 0 \pmod{p_3}$. Since $p_3$ is a primitive divisor of $U_4$, we have, from part (\ref{Cor:3}) of Corollary \ref{Cor:primdiv}, that $\alpha^2\equiv -1 \pmod{p_3}$. Then \[0\equiv A_3=\alpha^2-2\alpha+2\equiv -2\alpha+1 \pmod{p_3}.\] Clearly, $p_3\ne 2$, and so $\alpha\equiv 1/2 \pmod{p_3}$. Substituting this quantity back into $\alpha^2\equiv -1 \pmod{p_3}$ implies that $p_3=5$, and therefore $\alpha \equiv 3 \pmod{5}$. Unfortunately, no contradiction is achieved here. We use part (\ref{Cor:4}) of Corollary \ref{Cor:primdiv} to show in this situation that there is a second odd primitive divisor $q\ne 5$ of $U_4$. Then we can conclude that $A_3\not \equiv 0 \pmod{q}$. We consider two cases: $\alpha \equiv 3 \pmod{10}$ and $\alpha \equiv 8 \pmod{10}$.

First suppose that $\alpha\equiv 3 \pmod{10}$. Then $\alpha^2+1\equiv 2 \pmod{4}$. Le \cite{Le1999} proved that there are at most two pairs $(\alpha,n)$ of natural numbers such that \begin{equation}\label{Eq:Le}
\alpha^2+1=2\cdot 5^n.
\end{equation}
 Thus, the pairs $(3,1)$ and $(7,2)$ are the only solutions to equation (\ref{Eq:Le}). The solution $(7,2)$ is of no concern to us here, since $7\not\equiv 3 \pmod{10}$. Hence, when $\alpha>3$, there exists an odd prime $q\ne 5$ such that $\alpha^2+1\equiv 0 \pmod{q}$. To show that $q$ is indeed a primitive divisor of $U_4$, it is enough, by part (\ref{Cor:2}) of Corollary \ref{Cor:primdiv}, to show that $q$ does not divide either $\Phi_1(\alpha)$ or $\Phi_2(\alpha)$. But the only prime $q$ that can divide either $\Phi_1(\alpha)$ or $\Phi_2(\alpha)$, and also divide $\alpha^2+1$, is $q=2$. Recall that the case $\alpha=3$ is not an issue here since we are assuming that $\alpha\ne 2^c-1$.

 Next, suppose that $\alpha\equiv 8 \pmod{10}$. Then $\alpha^2+1$ is odd, and we need to examine the equation
 \begin{equation}\label{Eq:Leb}
\alpha^2+1=5^n.
\end{equation}
Lebesgue \cite{Lebesgue1850} proved that there exists at most one pair of natural numbers $(\alpha,n)$ that satisfies (\ref{Eq:Leb}). Thus $(2,1)$ is the only solution to (\ref{Eq:Leb}), and as above, $\alpha^2+1$ has a second odd primitive divisor $q\ne 5$.

Then, choosing $p_3$ to be the appropriate primitive divisor of $U_4$ so that $A_3\not \equiv 0 \pmod{p_3}$, we can apply the Chinese Remainder Theorem to the system of congruences $k\equiv -1/A_i \pmod{p_i}$,
to complete the proof in this case.

Now we turn our attention to the case when $\alpha=2^c-1$.
The Lucas pair $(2^c-1,1)$ is 2--defective, and so we cannot use the previous covering since $m_1=2$ there.
To avoid the 2-defective situation, we use a covering with minimum modulus 3:
\[\begin{array}{cccccccccccccccc}
  i & & 1&2&3&4&5&6&7&8&9&10&11&12&13&14\\ \hline
  r_i & & 0 & 0 & 1 & 5 & 6 & 3 & 10 & 4 & 11 & 2 & 7 & 35 & 25 & 55\\
 m_i & & 3 & 4 & 5 & 6 & 8 & 10 & 12 & 15 & 20 & 24 & 30 & 40 & 60 & 120.
  \end{array}\]
It is easy to check that this is indeed a covering.
First note that $\alpha-1=2^c-2\equiv 0 \pmod{2}$, so that $2$ is not a primitive divisor of $U_{m_i}$ for any $i$.
To ensure that $A_i\not \equiv 0 \pmod{p_i}$,
it is enough, by Corollary \ref{Cor:primdiv} part (\ref{Cor:3}),
 to show that
\begin{equation}\label{gcd condition}
 \gcd\left(\Phi_{m_i}(\alpha),A_i\right)\not \equiv 0 \pmod{p_i}.
\end{equation}
Tedious, but straightforward, arguments similar to the previous case show that (\ref{gcd condition}) is satisfied for all $i$ in this covering. Fortunately, no Diophantine equations must be considered here to show the existence of additional primitive divisors. Coverings with fewer congruences can be chosen with minimum modulus 3, but they all seem to incur the Diophantine considerations. Since the arguments in this case are similar to the previous case, we omit the details.
\end{proof}

\begin{remark}
  In the proof of Theorem \ref{Thm:main} we showed that if $\alpha\equiv 3 \pmod{5}$ and $\alpha\ne 2^c-1$, then $U_4(\alpha,1)$ has at least two distinct odd primitive divisors. This fact overlaps with a result of Schinzel \cite{Schinzel1962prim} when $\alpha$ is twice a square.
\end{remark}

\section{Generalization II}\label{Section:II}

In this section, we generalize Theorem \ref{Thm:Sierpinski} using an approach different from the one used in Section \ref{Section:I}. The main theorem here is:
\begin{thm}\label{Thm:general}
  There are infinitely many Lucas pairs $(\alpha,\beta)$, not produced by Theorem \ref{Thm:main}, for which there exist infinitely many positive integers $k$ such that \[k\left(U_n(\alpha,\beta)+(\alpha-\beta)^2\right)+1\]
  is 
   composite for all integers $n\ge 1$.
\end{thm}
The only rational Lucas pairs $(\alpha,\beta)$ that can be generated using the techniques in the proof of Theorem \ref{Thm:general} are such that $\alpha-\beta=1$, and thus the only conceivable overlap with Theorem \ref{Thm:main} is the Lucas pair $(2,1)$. Unfortunately, the algorithm, as described in the proof, does not directly capture this pair. However, a slight modification to the algorithm does the job (see Example (\ref{Ex:Sierpinski})), and so Theorem \ref{Thm:general} can, in some sense, be viewed as a generalization of Theorem \ref{Thm:Sierpinski}. Although primitive divisors were used successfully in the proof of Theorem \ref{Thm:main}, they are more difficult to harness when $\alpha$ and $\beta$ are not rational. For this reason, we abandon the use of primitive divisors in the proof of Theorem \ref{Thm:general}, in favor of a strategy that is somewhat opposite in nature. In the primitive-divisor situation, the primes we use (the primitive divisors themselves) depend on the particular values of $\alpha$ and $\beta$, while in the new approach, we start with a set of primes, and then construct the values of $\alpha$ and $\beta$. Although these methods produce rational, irrational, and nonreal Lucas pairs, depending on the covering used, there is an inherent weakness in the algorithm. Even allowing modifications to the algorithm, it seems that, in general, there is no way of determining ahead of time whether a particular Lucas pair can be captured by this procedure. In fact, there seem to be certain Lucas pairs that cannot be produced by these techniques (see Section \ref{Section:III}).

 The proof of Theorem \ref{Thm:general} is straightforward. Simply choose a particular covering, and use the algorithm to generate an explicit Lucas pair $(\alpha, \beta)$ that satisfies the conditions of the theorem. The algorithm is such that there are infinitely many choices from an arithmetic progression for values of $a$ and $b$, where $\alpha=(a+\sqrt{b})/2$, so that the algorithm automatically produces infinitely many values of $\alpha$ and $\beta$. Then there are infinitely many values of $k$ from an arithmetic progression that satisfy the conditions of the theorem for all values of $\alpha$ and $\beta$. 
In the proof of Theorem \ref{Thm:general}, we give a very specific version of the algorithm which can be used to generate irrational Lucas pairs. However, slight modifications will produce rational or nonreal Lucas pairs. We indicate these versions after the proof, and we provide examples in Section \ref{Section: Additional Examples}.

 \begin{proof}[Proof of Theorem \ref{Thm:general}]
We describe a version of the algorithm that will generate a Lucas pair; then we justify the steps; and finally, we use the algorithm with a particular covering to illustrate the process. Let $\left\{n\equiv r_i \pmod{m_i}\right\}$ be a covering with distinct moduli $m_i$, such that $p_i:=m_i+1$ is prime for all $i$.
For each $i$, we choose integers $a_i$ and $b_i$ according to the following prescription:
\begin{equation}\label{Eq: Menu}\begin{array}{rl}
1. &  \text{If $r_i=1$ or $r_i\equiv 0 \pmod{2}$, then let $a_i=0$ and $b_i=1$.}\\
2. & \text{If $r_i=3$, then let $a_i=0$ and $b_i=1$;}\\
  & \text{ unless $p_i=5$, in which case, let $a_i=1$ and $b_i=1$.}\\
3. & \text{If $r_i\ge 5$ and $r_i\equiv 1 \pmod{2}$, then let $a_i=0$ and $b_i=4$.}\\
\end{array}
\end{equation}
Then, use the Chinese remainder theorem to solve the two systems of congruences \[x\equiv a_i \pmod{p_i}, \quad x\equiv 1 \pmod{2}, \mbox{ \quad and }\vspace*{-.2in}\\\] \begin{equation}\label{Eq:Systems}\vspace*{-.2in}\\\end{equation}
\[y\equiv b_i \pmod{p_i}, \quad y\equiv 1 \pmod{4}.\] Let $a$ and $b$ be respective solutions to the systems in (\ref{Eq:Systems}), and let $\alpha=(a+\sqrt{b})/2$ and $\beta=(a-\sqrt{b})/2$. At this juncture, we must verify that $(\alpha,\beta)$ is a legitimate Lucas pair. Clearly, $\alpha+\beta=a\in \Z$. Observe that $\alpha\beta=(a^2-b)/4$. Since $a\equiv 1 \pmod{2}$ and $b\equiv 1 \pmod{4}$, we have that $\alpha\beta\in \Z$. Next, since $\gcd(\alpha+\beta,\alpha\beta)=1$ if and only if $\gcd(a,b)=1$, we need to be able to select solutions $a$ and $b$ of (\ref{Eq:Systems}) that are relatively prime. To accomplish this task, fist solve for $a$ in the first system of (\ref{Eq:Systems}), and then add additional congruences, if necessary, to the second system in (\ref{Eq:Systems}) to guarantee that $\gcd(a,b)=1$. Then we must check that $\alpha/\beta$ is not a root of unity. Finally, use the Chinese remainder theorem to solve the system of congruences $k\equiv -1/A_i \pmod{p_i}$, where \[A_i:=\frac{\alpha^{r_i}-\beta^{r_i}}{\alpha-\beta}+(\alpha-\beta)^2.\] Then we claim that $k(U_n(\alpha,\beta)+(\alpha-\beta)^2)+1$ is composite for all $n\ge 1$. In fact, we have that
\begin{equation}\label{Eq: Claim}\begin{array}{c}
k(U_n(\alpha,\beta)+(\alpha-\beta)^2)+1 \equiv 0 \pmod{p_i},\\
\mbox{when $n\equiv r_i \pmod{m_i}$}.\\
\end{array}
\end{equation}

To prove that (\ref{Eq: Claim}) is true, we verify the validity of the steps of the algorithm, and show that $A_i\not \equiv 0 \pmod{p_i}$ for all $i$. First note that the conditions in (\ref{Eq: Menu}) guarantee that $b_i\not \equiv 0 \pmod{p_i}$ for all $i$. Consequently, $b\not \equiv 0 \pmod{p_i}$ and $\alpha-\beta\not \equiv 0 \pmod{p_i}$ for all $i$.
For each $i$, let $\alpha_i=(a_i+\sqrt{b_i})/2$, $\beta_i=(a_i-\sqrt{b_i})/2$, and \[\bar{A_i}:=\frac{\alpha_i^{r_i}-\beta_i^{r_i}}{\alpha_i-\beta_i}+(\alpha_i-\beta_i)^2=\frac{\alpha_i^{r_i}-\beta_i^{r_i}}{\alpha_i-\beta_i}+b_i.\] Since $b_i$ is a square modulo $p_i$, and $m_i=p_i-1$ for all $i$, it follows from Fermat's little theorem (even if $\alpha\equiv 0 \pmod{p_i}$ or $\beta\equiv 0 \pmod{p_i}$, which could happen if $n\equiv 3 \pmod{4}$ is a congruence in the covering) that
\[U_n(\alpha,\beta)+(\alpha-\beta)^2 \equiv A_i \equiv \bar{A_i} \pmod{p_i},\]
when $n\equiv r_i \pmod{m_i}$. 
First assume that $a_i=0$. Then straightforward calculations give
\begin{equation}\label{Eq:A_iconditions}\begin{array}{cl}\bar{A_i}&=\ds \frac{\left(\sqrt{b_i}\right)^{r_i-1}\left(1-\left(-1\right)^{r_i}\right)}{2^{r_i}}+b_i\\\\
&=\left\{\begin{array}{cl}
\ds b_i & \mbox{if $r_i\equiv 0 \pmod{2}$}\\\\
\ds \frac{\left(\sqrt{b_i}\right)^{r_i-1}}{2^{r_i-1}}+b_i & \mbox{if $r_i\equiv 1 \pmod{2}$.}\end{array} \right. \end{array}\end{equation}
We refer to the menu (\ref{Eq: Menu}). It is clear that $\bar{A_i}\not \equiv 0 \pmod{p_i}$ in (\ref{Eq:A_iconditions}) when $r_i\equiv 0 \pmod{2}$, since $b_i=1$. When $r_i\equiv 1 \pmod{2}$, there are three cases to consider. If $r_i=1$, then $b_i=1$, and so $\bar{A_i}=2\not \equiv 0 \pmod{p_i}$. If $r_i=3$, then $b_i=1$ and $\bar{A_i}=5/4\not \equiv 0 \pmod{p_i}$, since $p_i\ne 5$. Next, if $r_i\ge 5$, then $p_i\ge 7$. Then, since $b_i=4$ from (\ref{Eq: Menu}), we have that $\bar{A_i}=5\not \equiv 0 \pmod{p_i}$.

Now assume that $a_i=1$. Then $r_i=3$, $p_i=5$, and $b_i=1$ from (\ref{Eq: Menu}). In this case, either $\alpha=0$ and $\beta=1$, or $\alpha=1$ and $\beta=0$. In either situation, we have that $\bar{A_i}=1\not \equiv 0 \pmod{5}$.

 We finish the proof with an example.
Consider the covering
\[\begin{array}{cccccccccccccccc}
  i & & 1&2&3&4&5&6&7&8&9&10&11&12&13\\ \hline
  r_i & & 0 & 1 & 5 & 7 & 3 & 7 & 1 & 19 & 55 & 31 & 139 & 13 & 103\\
 m_i & & 2 & 4 & 6 & 10 & 12 & 18 & 30 & 36 & 60 & 108 & 180 & 270 & 540.
  \end{array}\]
Applying the algorithm to this situation gives:
\[\alpha=\frac{57735618045574774305+\sqrt{41575375575250122841}}{2}, \quad \mbox{and}\] \[k=37170467875892126822.\]
The first three terms of the sequence $k(U_n(\alpha,\beta)+(\alpha-\beta)^2)+1$, in factored form, with $p_i$ in bold, are given in Table \ref{Table:Fac1}.
\begin{table}[h]
\centering
\begin{tabular}{|c|c|}\hline
$n$ & $k(U_n(\alpha,\beta)+(\alpha-\beta)^2)+1$\\ \hline\hline
1 & ${\bf 5}^4\cdot 7\cdot 11\cdot 19\cdot \cdot 31 \cdot 37 \cdot 61 \cdot 109 \cdot 181 \cdot 271 \cdot 541$\\
& $ \cdot 22409 \cdot 372668347052399$ \\ \hline
2 & ${\bf 3}\cdot 24691\cdot 49835109933522332988999783635863781$\\ \hline
3 & $7\cdot 11\cdot {\bf 13} \cdot 19\cdot 37 \cdot 61 \cdot 109 \cdot 181 \cdot 271 \cdot 541 \cdot 2127299$\\
 & $ \cdot 2258992037077 \cdot 155744538873346913742671$\\ \hline
\end{tabular}
\caption{{\footnotesize Factored Terms of $k(U_n(\alpha,\beta)+(\alpha-\beta)^2)+1$}}
\label{Table:Fac1}
\end{table}\\\end{proof}

In general, the algorithm given in the proof of Theorem \ref{Thm:general} will produce an irrational Lucas pair. However, if all residues in the covering are even, or if the only odd residues that appear in the covering are $r_i=1$, then $b_i=1$ for all $i$, and we can take $b=1$. The algorithm generates a rational Lucas pair in this situation (see Example \ref{Ex:Rational}). Also, it is easy to see that there is room for modification of the algorithm. For example, we chose $a_i=0$, for most values of $i$, since it is easier to prove that $A_i$ is invertible modulo $p_i$ in that situation. But to produce the Lucas pair $(2,1)$, we can choose all $a_i=3$ and all $b_i=1$, with an appropriate covering (see Example \ref{Ex:Sierpinski}). To generate a nonreal Lucas pair, we can let $b$ be a negative value in the arithmetic progression produced by solving the second system in (\ref{Eq:Systems}) (see Example \ref{Ex:nonreal}).
  Other modifications to the algorithm are possible, depending on the chosen covering, but these modifications could result in a more complicated set of conditions for $\bar{A_i}$ to be invertible modulo $p_i$.

\subsection{Additional Examples}\label{Section: Additional Examples}
This section contains some more examples illustrating the algorithm used in the proof of Theorem \ref{Thm:general}, and some modified versions of it. To keep the numbers reasonably small, we have chosen coverings in which the maximum modulus is 180 and the greatest common divisor of the moduli is 360.
\begin{ex}\label{Ex:Rational}
This example shows how the algorithm in Theorem \ref{Thm:general} can be used to produce a rational Lucas pair. We use the covering:
 \[\begin{array}{cccccccccccccccc}
  i & & 1&2&3&4&5&6&7&8&9&10&11&12\\ \hline
  r_i & & 1 & 0 & 0 & 0 & 10 & 8 & 2 & 2 & 14 & 38 & 50 & 86\\
 m_i & & 2 & 4 & 6 & 10 & 12 & 18 & 30 & 36 & 40 & 60 & 72 & 180.
  \end{array}\]
  Observe that the only odd residue is $r_1=1$.
 Applying the algorithm gives
  \[\alpha=5406640414743068,\quad \beta=5406640414743067\quad \mbox{and}\] \[k=3604426943162044.\]
The first three terms of the sequence $k(U_n(\alpha,\beta)+(\alpha-\beta)^2)+1$, in factored form, with $p_i$ in bold, are given in Table \ref{Table:Fac2}.
\begin{table}[h]
\centering
\begin{tabular}{|c|c|}\hline
$n$ & $k(U_n(\alpha,\beta)+(\alpha-\beta)^2)+1$\\ \hline\hline
1 & ${\bf 3}^2\cdot 1708529 \cdot 468814849$ \\ \hline
2 & $5\cdot 7 \cdot 11\cdot 13\cdot 17 \cdot 19 \cdot {\bf 31} \cdot 37 \cdot 41 \cdot 61 \cdot 73 \cdot 181 \cdot 2179 \cdot 62143\cdot 4697417$\\ \hline
3 & ${\bf 3}\cdot 8707 \cdot 15328919 \cdot 83120546683 \cdot 9497356395852767786266693$\\ \hline
\end{tabular}
\caption{{\footnotesize Factored Terms of $k(U_n(\alpha,\beta)+(\alpha-\beta)^2)+1$}}\vspace*{.11in}
\label{Table:Fac2}
\end{table}
\end{ex}

\begin{ex}\label{Ex:nonreal}
This example shows how to produce a nonreal Lucas pair. We start with the covering:
 \[\begin{array}{ccccccccccccccc}
  i & & 1&2&3&4&5&6&7&8&9&10&11&12\\ \hline
  r_i & & 0 & 3 & 1 & 3 & 9 & 11 & 17 & 5 & 1 & 5 & 53 & 89\\
  m_i & & 2 & 4 & 6 & 10 & 12 & 18 & 30 & 36 & 40 & 60 & 72 & 180.
  \end{array}\]
We take the smallest positive value of $b$ produced by the algorithm and subtract the least common multiple of the moduli in the second system in (\ref{Eq:Systems}) to get the negative value of $b=-10777658998435559$. Then
  \[\alpha=\frac{6487968497691681+\sqrt{-10777658998435559}}{2}, \quad \mbox{and}\] \[k=1314262889709437.\]
The first three terms of the sequence $k(U_n(\alpha,\beta)+(\alpha-\beta)^2)+1$, in factored form, with $p_i$ in bold, are given in Table \ref{Table:Fac3}.

\begin{table}[h]
\centering
\begin{tabular}{|c|c|}\hline
$n$ & $k(U_n(\alpha,\beta)+(\alpha-\beta)^2)+1$\\ \hline\hline
1 & $-5\cdot {\bf 7} \cdot 13\cdot 19 \cdot 31 \cdot 37 \cdot 41 \cdot 61 \cdot 73 \cdot 181 \cdot 499 \cdot 51131 \cdot 1694253179$\\ \hline
2 & $-{\bf 3}\cdot 5 \cdot 557 \cdot 3319249 \cdot 203292762260868131903$\\ \hline
3 & ${\bf 5}\cdot 11 \cdot 13\cdot 19 \cdot 31 \cdot 37 \cdot 61 \cdot 73 \cdot 181 \cdot 50257221163$\\
& $ \cdot 65736741235555550052593$\\ \hline
\end{tabular}
\caption{{\footnotesize Factored Terms of $k(U_n(\alpha,\beta)+(\alpha-\beta)^2)+1$}}\vspace*{.11in}
\label{Table:Fac3}
\end{table}
\end{ex}

\begin{ex}\label{Ex:Sierpinski}
This example shows how the algorithm in Theorem \ref{Thm:general} can be modified to capture the Lucas pair $(\alpha,\beta)=(2,1)$ and give an alternate proof of Theorem \ref{Thm:Sierpinski}. We start with the covering:

 \[\begin{array}{cccccccccccccccc}
  i & & 1&2&3&4&5&6&7&8&9&10&11&12\\ \hline
  r_i & & 0 & 1 & 1 & 5 & 11 & 15 & 9 & 3 & 23 & 51 & 27 & 27\\
 m_i & & 2 & 4 & 6 & 10 & 12 & 18 & 30 & 36 & 40 & 60 & 72 & 180.
  \end{array}\]
 We modify the algorithm by choosing $a_i=3$ and $b_i=1$ for all $i$, and in addition, we replace the congruence $x\equiv 1 \pmod{2}$ in (\ref{Eq:Systems}) with the congruence $x\equiv 3 \pmod{4}$. The algorithm then produces $\alpha=2$, $\beta=1$, and $k=9579495527398457$.
The first three terms of the sequence $k(U_n(\alpha,\beta)+(\alpha-\beta)^2)+1$, in factored form, with $p_i$ in bold, are given in Table \ref{Table:Fac5}.
\begin{table}[h]
\centering
\begin{tabular}{|c|c|}\hline
$n$ & $k(U_n(\alpha,\beta)+(\alpha-\beta)^2)+1$\\ \hline\hline
1 & ${\bf 5}\cdot 7^2 \cdot 43\cdot 1459 \cdot 607147 \cdot 2053$ \\ \hline
2 & ${\bf 3}\cdot 907 \cdot 14082316100549$\\ \hline
3 & ${\bf 37}\cdot 41\cdot 1399 \cdot 36110153179$\\ \hline
\end{tabular}
\caption{{\footnotesize Factored Terms of $k(U_n(\alpha,\beta)+(\alpha-\beta)^2)+1$}}\vspace*{.11in}
\label{Table:Fac5}
\end{table}
\begin{remark}
  The Sierpi\'{n}ski number $k$ produced by this procedure in Example (\ref{Ex:Sierpinski}) is considerably smaller than the smallest Sierpi\'{n}ski number generated in Sierpi\'{n}ski's original proof.
\end{remark}
\end{ex}

\section{Another Approach} \label{Section:III}

The algorithm used to prove Theorem \ref{Thm:general} (and any modification) appears to be too restrictive to produce certain Lucas pairs. For example, it seems unlikely that the famous Lucas pair $\left((1+\sqrt{5})/2,(1-\sqrt{5})/2\right)$, which generates the Fibonacci sequence $\left\{F_n\right\}$, can be captured using this algorithm. One reason for this is that Fermat's little theorem does not apply if $5$ is not a square modulo $p_i=m_i+1$. However, constructing a covering by replacing such ``bad" moduli with distinct moduli $m_i$, such that $m_i+1$ is prime, and for which $5$ is a square modulo $m_i+1$, is most certainly a difficult task at best, and it is quite plausible that it is impossible. We have been unsuccessful in our attempts to construct such a covering.

The approach used in this section is quite different from the methods used in the previous sections. Instead of directly using primitive divisors, or a covering where each modulus is one less than a prime, we exploit the well-known fact that Lucas sequences $U_n$ are periodic modulo any prime \cite{Rec}. The idea is to construct a covering where each modulus is a period of $U_n$ modulo some prime. If $U_n$ has period $m$ modulo the prime $p$, then $U_m\equiv 0 \pmod{p}$, but $p$ might or might not be a primitive divisor of $U_m$. However, there is always a least positive integer $a(p)$, called the {\it restricted period} \cite{Rec} of $U_n$ modulo $p$, such that $p$ is a primitive divisor of $U_{a(p)}$. So, we are using primitive divisors in some sense, but certainly not in the traditional way. Just as $U_n$ may have more than one primitive divisor, $U_n$ can have the same period modulo more than one prime. Thus, our covering can have repeated moduli, and we make use of this phenomenon to establish Theorem \ref{Thm:Fib}. However, constructing the covering is still somewhat tricky, since, depending on the particular sequence $U_n$, there can be many positive integers $m$ for which there is no prime $p$ such that $U_n$ has period $m$ modulo $p$. For example, the only odd period for $\left\{F_n\right\}$ is $m=3$. Although we are unable to determine, in general, when this process will be successful, we illustrate that the method does work in certain situations by establishing that the procedure is successful in the case of the Fibonacci sequence $\left\{F_n\right\}$. Helpful in the construction of the covering here is the fact, which follows from a result of Lengyel \cite{Lengyel1995}, that given any even number $m\not \in \left\{2,4,6,12,24 \right\}$, there exists at least one prime $p$ such that the period of $\left\{F_n\right\}$ modulo $p$ is $m$.
 The main result of this section is:
\begin{thm}\label{Thm:Fib}
Let $\left\{F_n\right\}$ denote the Fibonacci sequence, defined recursively by $F_0=0$, $F_1=1$, and $F_n=F_{n-1}+F_{n-2}$, for $n\ge 2$. Then there exist infinitely many positive integers $k$ such that the sequence
 $k(F_n+5)+1$ is composite for all integers $n\ge 1$.
\end{thm}
\begin{proof}
We use a slightly different format to present the covering $\C$ here.
 The 133 elements of $\C$ are indicated by ordered triples $(r_i,m_i,p_i)$, where the congruence in the covering corresponding to this ordered triple is $n\equiv r_i \pmod{m_i}$. The prime $p_i$ in the ordered triple is the prime such that $\left\{F_n\right\}$ modulo $p_i$ has period $m_i$.
Then, for each congruence $n\equiv r_i \pmod{m_i}$ in the covering, we have that $F_n\equiv F_{r_i} \pmod{p_i}$. We must then check that $F_{r_i}+5 \not \equiv 0 \pmod{p_i}$. Finally, we use the Chinese remainder theorem to solve for $k$ in the system of 133 congruences $k\equiv -1/(F_{r_i}+5)  \pmod{p_i}$.
The covering we use here is:
$\C=\{(0, 3, 2), (0, 8, 3), (1, 10, 11), (6, 14, 29), (6, 16, 7), (5, 18, 19),\\
  (3, 20, 5), (2, 28, 13), (19, 30, 31), (12, 32, 47), (29, 36, 17),\\
  (27, 40, 41), (22, 42, 211), (20, 48, 23)  , (5, 50, 101)  , (45, 50, 151),\\
  (35, 54, 5779)  , (18, 56, 281), (37, 60, 61), (0, 70, 71), (12, 70, 911),\\
  (47, 72, 107), (14, 80, 2161), (10, 84, 421), (89, 90, 181), (85, 90, 541),\\
  (92, 96, 1103), (13, 100, 3001), (53, 108, 53), (17, 108, 109),\\
  (42, 112, 14503), (7, 120, 2521), (40, 126, 1009), (124, 126, 31249),\\
  (42, 140, 141961), (100, 144, 103681), (85, 150, 12301), (115, 150, 18451),\\
  (78, 160, 1601), (46, 160, 3041), (50, 162, 3079), (140, 162, 62650261),\\
  (122, 168, 83), (50, 168, 1427), (73, 180, 109441), (75, 200, 401),\\
  (175, 200, 570601), (110, 210, 21211), (196, 210, 767131),\\
  (4, 216, 11128427), (158, 224, 10745088481), (193, 240, 241),\\
  (133, 240, 20641), (82, 252, 35239681), (29, 270, 271), (17, 270, 811),\\
  (119, 270, 42391), (209, 270, 119611), (154, 280, 12317523121),\\
  (28, 288, 10749957121), (25, 300, 230686501), (124, 324, 2269),\\
  (232, 324, 4373), (148, 324, 19441), (26, 336, 167), (206, 336, 65740583),\\
  (98, 350, 54601), (168, 350, 560701), (28, 350, 7517651),\\
  (238, 350, 51636551), (133, 360, 10783342081), (88, 378, 379),\\
  (130, 378, 85429), (214, 378, 912871), (52, 378, 1258740001),\\
  (393, 400, 9125201), (153, 400, 5738108801), (278, 420, 8288823481),\\
  (292, 432, 6263), (196, 432, 177962167367), (215, 450, 221401),\\
  (35, 450, 15608701), (335, 450, 3467131047901),\\
   (446, 480, 23735900452321), (268, 504, 1461601), (436, 504, 764940961),\\
    (107, 540, 1114769954367361), (306, 560, 118021448662479038881),\\
   (73, 600, 601), (433, 600, 87129547172401), (92, 630, 631),\\
    (476, 630, 1051224514831), (260, 630, 1983000765501001),\\
   (340, 648, 1828620361), (364, 648, 6782976947987),\\
    (638, 672, 115613939510481515041), (658, 700, 701),\\
   (474, 700, 17231203730201189308301), (13, 720, 8641),\\
  (515, 720, 13373763765986881), (700, 756, 38933),\\
  (472, 756, 955921950316735037), (715, 800, 124001), (315, 800, 6996001),\\
  (782, 800, 3160438834174817356001), (742, 810, 1621), (94, 810, 4861),\\
  (580, 810, 21871), (418, 810, 33211), (256, 810, 31603395781),\\
  (34, 810, 7654861102843433881), (194, 840, 721561),\\
  (266, 840, 140207234004601), (508, 864, 3023), (412, 864, 19009),\\
  (14, 864, 447901921), (686, 864, 48265838239823),\\
  (242, 900, 11981661982050957053616001), (46, 1008, 503),\\
  (494, 1008, 4322424761927), (830, 1008, 571385160581761),\\
   (302, 1050, 1051), (722, 1050, 9346455940780547345401),\\
  (512, 1050, 14734291702642871390242051), (590, 1080, 12315241),\\
  (950, 1080, 100873547420073756574681), (942, 1120, 6135922241),\\
  (270, 1120, 164154312001), (750, 1120, 13264519466034652481),\\
  (428, 1134, 89511254659), (680, 1134, 1643223059479),\\
  (806, 1134, 68853479653802041437170359),\\
   (1058, 1134, 5087394106095783259)\}$.\\ The smallest positive value of $k$ found using the Chinese remainder theorem has 949 digits.

We do not give the first three terms of the sequence $k(F_n+5)+1$ in factored form since they are too large.
\end{proof}

\begin{remark}
  As far as the author knows, the covering $\C$ used in the proof of Theorem \ref{Thm:Fib} is the first time a covering, all of whose moduli are periods of the Fibonacci sequence modulo some prime, has appeared in the literature. The periods of the Fibonacci sequence are also known as the Pisano periods. 
\end{remark}
\section{A Nonlinear Variation}

Given a nonlinear polynomial $f(k)$, we can ask whether there exist infinitely many positive integers $k$ such that $f(k)\cdot 2^n+1$ is composite for all integers $n\ge 1$. With $f(k)=k^r$, Chen \cite{Chen4} proved that the answer is affirmative for $r\not \equiv 0,4,6,8 \pmod{12}$. Using a different approach, Filaseta, Finch and Kozek \cite{FFK} have been able to lift Chen's restriction on $r$. More recently, Finch, Harrington and the author (unpublished manuscript) have established a similar result with $f(k)=k^r+1$, when $r$ is not divisible by 8 or 17449. We should point out that Chen, and Filaseta, Finch and Kozek also addressed other concerns in their respective papers. For example, these authors actually showed that each composite term in the sequence has at least two distinct prime divisors.

 We end this paper with a nonlinear variation using Lucas sequences.
\begin{thm}\label{Thm:nonlinear}
  Let $m=\prod_{i=1}^{11}p_i$, where the $p_i$ are given in Table \ref{Table:CovNonlinear}, and let $\alpha\equiv 5 \pmod{m}$ be a positive integer.
Then there exist infinitely many positive integers $k$ such that
\[k^2\left(U_n(\alpha,1)+(\alpha-1)^2\right)+1\]
is composite for all integers $n\ge 1$.
\end{thm}
\begin{proof}
  The covering $\left\{n\equiv r_i \pmod{m_i}\right\}$ we use here is given in Table \ref{Table:CovNonlinear}.
  \begin{table}[h]
\centering
\begin{tabular}{ccccccccccccc}
$i$ & & 1&2&3&4&5&6&7&8&9&10&11\\ \hline
$r_i$ && 1 & 1 & 0 & 1 & 2 & 6 & 0 & 12 & 14 & 18 & 0\\ 
$m_i$ && 2 & 3 & 4 & 6 & 8 & 9 & 12 & 18 & 24 & 36 & 72\\ 
$p_i$ && 3 & 31 & 13 & 7 & 313 & 19 & 601 & 5167 & 390001 & 37 & 73\\ 
\end{tabular}
\caption{{\footnotesize The covering with the primitive divisors $p_i$ of $U_{m_i}$}}\vspace*{.11in}
\label{Table:CovNonlinear}
\end{table}
The prime $p_i$ is a primitive divisor of $U_{m_i}$.
  For each $i$, let $A_i=U_{r_i}+(\alpha-1)^2$, so that $U_n+(\alpha-1)^2 \equiv A_i \pmod{p_i}$ when $n\equiv r_i \pmod{m_i}$. It is then easy to verify that $A_i\not \equiv 0 \pmod{p_i}$, and that $-1/A_i$ is a square modulo $p_i$ for all $i$. We solve each of the congruences $k^2\equiv -1/A_i \pmod{p_i}$, and choose a solution $s_i$. This gives a system of congruences $k\equiv s_i \pmod{p_i}$, and we can apply the Chinese remainder theorem to this system to find infinitely values of $k$. The smallest positive value of $k$ produced by this process is $k=117050073288612071969896$.
  The first three terms of the sequence $k^2(U_n(5,1)+16)+1$, in factored form, with $p_i$ in bold, are given in Table \ref{Table:Fac6}.
 \begin{table}[h]
\centering
\begin{tabular}{|c|c|}\hline
$n$ & $k^2(U_n(5,1)+16)+1$\\ \hline\hline
1 & ${\bf 3}\cdot 7 \cdot 23\cdot 31 \cdot 53 \cdot 199 \cdot 431 \cdot 3132881 \cdot 3384559190303$\\
  & $ \cdot 322723988351788951$\\ \hline
2 & ${\bf 313}\cdot 571 \cdot 853 \cdot 1459 \cdot 5931337 \cdot 336267671 \cdot 18194404469$\\
  & $ \cdot 37342701311$\\ \hline
3 & ${\bf 3}^3\cdot 17\cdot 377843803411203610837 \cdot 3712925610260096762131991$\\ \hline
\end{tabular}
\caption{{\footnotesize Factored Terms of $k^2(U_n(5,1)+16)+1$}}\vspace*{.11in}
\label{Table:Fac6}
\end{table}

\end{proof}
\begin{remark}
 The computer calculations and verifications needed in this paper were done using either MAGMA or Maple.
\end{remark}


















\end{document}